\documentclass[11pt]{amsart}
\usepackage{verbatim, graphicx, rotating}
\usepackage{mdwlist}
\usepackage{enumerate,mdwlist}
\usepackage{url}
\usepackage{amsthm}
\usepackage{hyperref}
\usepackage{multirow}

\usepackage {amssymb, amsmath, mathrsfs}

\input xy
\xyoption{all}
\input epsf

\newlength{\tabwidth}
\newlength{\tabheight}
\newlength{\tabrule}
\newlength{\tabwidthx}
\newlength{\tabheightx}

\def\gentabbox#1#2#3#4{\vbox to \tabheight{\setlength{\tabrule}{#3}%
  \setlength{\tabwidthx}{#1\tabwidth}\addtolength{\tabwidthx}{\tabrule}%

\setlength{\tabheightx}{#2\tabheight}\addtolength{\tabheightx}{-\tabheight}%
  \hbox to #1\tabwidth{%
    \hspace{-0.5\tabrule}\rule{\tabrule}{#2\tabheight}\hspace{-\tabrule}%
    \vbox to #2\tabheight{\hsize=\tabwidthx%
      \vspace{-0.5\tabrule}\hrule width\tabwidthx height\tabrule%
      \vspace{-0.5\tabrule}\vfil%
      \hbox to \tabwidthx{\hss#4\hss}%
        \vfil\vspace{-0.5\tabrule}%
      \hrule width\tabwidthx height\tabrule\vspace{-0.5\tabrule}}%
    \hspace{-\tabrule}\rule{\tabrule}{#2\tabheight}\hspace{-0.5\tabrule}}%
  \vspace{-\tabheightx}}}
\def\genblankbox#1#2{\vbox to \tabheight{\vfil\hbox to
#1\tabwidth{\hfil}}}
\def\tabbox#1#2#3{\gentabbox{#1}{#2}{0.4pt}{\strut #3}}

\catcode`\:=13 \catcode`\.=13 \catcode`\;=13
\catcode`\>=13 \catcode`\^=13
\def:#1\\{\hbox{$#1$}}
\def.#1{\tabbox{1}{1}{$#1$}}
\def>#1{\tabbox{2}{1}{$#1$}}
\def^#1{\tabbox{1}{2}{$#1$}}
\def;{\genblankbox{1}{1}\relax}
\catcode`\:=12 \catcode`\.=12 \catcode`\;=12
\catcode`\>=12 \catcode`\^=7

\newcommand{\field}{\mathbb}
\newcommand{\liealgebra}{\mathfrak}
\newcommand{\la}{\liealgebra}

\newcommand{\C}{{\field C}}

\newcommand{\Z}{{\field Z}}

\newcommand{\Q}{{\field Q}}
\newcommand{\PP}{{\field P}}

\renewcommand{\b}{\liealgebra b}

\newcommand{\n}{{\la n}}

\newcommand{\ga}{\alpha}

\newtheorem{prop}{Proposition}[section]

\newtheorem{theorem}[prop]{Theorem}

\newtheorem{corollary}[prop]{Corollary}
\newtheorem{proposition}[prop]{Proposition}

\theoremstyle{definition}

\newtheorem{remark}[prop]{Remark}
\newtheorem{example}[prop]{Example}

\newtheorem{definition}[prop]{Definition}

\newcommand{\frS}{\mathfrak{S}}

\newcommand{\frs}{\mathfrak{s}}
\newcommand{\frt}{\mathfrak{t}}

\newcommand{\caO}{\mathcal{O}}

\newcommand{\xx}{\mathbf x}
\newcommand{\yy}{\mathbf y}
\newcommand{\zz}{\mathbf z}

\newcommand{\mb}{\mathbf }

\begin{document}
\title[Wonderful symmetric varieties and Schubert polynomials]
{Wonderful symmetric varieties and Schubert polynomials}

\author{Mahir Bilen Can}
\address{Department of Mathematics, Tulane University, New Orleans LA 70118 USA}
\email{mcan@tulane.edu}

\author{Michael Joyce}
\address{Department of Mathematics, Tulane University, New Orleans LA 70118 USA}
\email{mjoyce3@tulane.edu}

\author{Benjamin Wyser}
\address{Department of Mathematics, University of Illinois at Urbana-Champaign, Urbana IL 61801 USA}
\email{bwyser@illinois.edu}

\date{\today}



\begin{abstract}
Extending results of \cite{Wyser-13-TG}, we determine formulas for the equivariant cohomology classes of closed orbits of certain families of spherical subgroups of $GL_n$ on the flag  variety $GL_n/B$.  Putting this together with a slight extension of the results of \cite{Can-Joyce-Wyser-15}, we arrive at a family of polynomial identities which show that certain explicit sums of Schubert polynomials factor as products of linear forms.
\end{abstract}
\maketitle

\section{Introduction}
Suppose that $G$ is a connected reductive algebraic group over $\C$. Suppose that $B \supseteq T$ are a Borel subgroup and a maximal torus of $G$, respectively, $W$ is the Weyl group, and let $\mathfrak{t}$ denote the Lie algebra of $T$. By a classical theorem of Borel \cite{Borel-53}, the cohomology ring of $G/B$ with rational coefficients is isomorphic to the coinvariant algebra $\mathbb{Q}[\mathfrak{t}^*]/I^W$, where $I^W$ denotes the ideal generated by homogeneous $W$-invariant polynomials of positive degree. Any subvariety $Y$ of $G/B$ defines a cohomology class $[Y]$ in $H^*(G/B)$.  It is then natural to ask for a polynomial in $\mathbb{Q}[\mathfrak{t}^*]$ which represents $[Y]$.  In this paper, for certain families of subvarieties of certain $G/B$, we approach and answer this question in two different ways.  Relating the two answers leads in the end to our main result, Theorem \ref{thm:sum-equals-product}, which, roughly stated, says that certain non-negative linear combinations of Schubert polynomials factor completely into linear forms.

Our group of primary interest is $\mb{G}= \mb{GL}_n$,  with $\mb{B}$ its Borel subgroup of lower-triangular matrices, and $\mb{T}$ its maximal torus of diagonal matrices. In this case, there is a canonical basis $x_1, \dots, x_n$ of $\mathfrak{\mathbf{t}}^*$ that correspond to the Chern classes of the tautological quotient line bundles on the variety of complete flags $\mb{G}/\mb{B}$. Let $\mb{Z}_n$ denote the center of $\mb{GL}_n$, consisting of diagonal scalar matrices.  Let $\mb{O}_n$ denote the orthogonal subgroup of $\mb{GL}_n$, and let $\mb{Sp}_{2n}$ denote the symplectic subgroup of $\mb{GL}_{2n}$.  Denote by $\mb{GO}_n$ (resp. $\mb{GSp}_{2n}$) the central extension $\mb{Z}_n \mb{O}_n$ (resp. $\mb{Z}_{2n} \mb{Sp}_{2n}$). For any ordered sequence of positive integers $\mu=(\mu_1,\hdots,\mu_s)$ that sum to $n$,
$\mb{GL}_n$ has a Levi subgroup $\mb{L}_{\mu} := \mb{GL}_{\mu_1} \times \hdots \times \mb{GL}_{\mu_s}$, as well as a parabolic subgroup $\mb{P}_{\mu}=\mb{L}_{\mu} \ltimes \mb{U}_{\mu}$ containing $\mb{B}$, where $\mb{U}_{\mu}$ denotes the unipotent radical of $\mb{P}_{\mu}$.

The subgroup
$$\mb{H}_{\mu}:=(\mb{GO}_{\mu_1} \times \hdots \times \mb{GO}_{\mu_s} ) \ltimes \mb{U}_{\mu}$$
of $\mb{GL}_n$ is \textit{spherical}, meaning that it acts on $\mb{GL}_n/\mb{B}$ with finitely many orbits.  Moreover, there is a unique \textit{closed} $\mb{H}_{\mu}$-orbit $\mb{Y}_{\mu}$ on $\mb{GL}_n/\mb{B}$, which is our object of primary interest.

The reason for our interest in this family of orbits is that they correspond to the closed $\mb{B}$-orbits on the various $\mb{G}$-orbits of the \textbf{wonderful compactification} of the homogeneous space $\mb{GL}_n/\mb{GO}_n$.  This homogeneous space is affine and symmetric, and it is classically known as the space of smooth quadrics in $\PP^{n-1}$.  Its wonderful compactification, classically known as the variety of complete quadrics \cite{Semple-48, Laksov-85}, is a $G$-equivariant projective embedding $\mb{X}$ which contains it as an open, dense $\mb{G}$-orbit, and whose boundary has particularly nice properties. (We recall the definition of the wonderful compactification in Section \ref{sec:bg-wc}.)

It turns out that, with minor modifications, our techniques apply also to the wonderful compactification $\mb{X}'$ of the space $\mb{GL}_{2n}/\mb{GSp}_{2n}$, which parameterizes non-degenerate skew-symmetric bilinear forms on $\C^{2n}$, up to scalar.
Letting $\mb{G} = \mb{GL}_{2n}$ in this case, the $\mb{G}$-orbits on $\mb{X'}$ are again parametrized by compositions $\mu=(\mu_1,\hdots,\mu_s)$ of $n$; note that this is of course equivalent to parametrizing them by compositions of $2n$ with each part being even.  Each $\mb{G}$-orbit has the form $\mb{G}/\mb{H}_{\mu}'$, with
$$\mb{H}_{\mu}':=(\mb{GSp}_{2\mu_1} \times \hdots \times \mb{GSp}_{2\mu_s} ) \ltimes \mb{U}_{\mu},$$
a spherical subgroup which again acts on $\mb{GL}_{2n} / \mb{B}$ with a unique closed orbit $\mb{Y}_{\mu}'$.

Let us consider two ways in which one might try to compute a polynomial representative of $[\mb{Y}_{\mu}]$ (or $[\mb{Y}_{\mu}']$). For the first, note that $\mb{Y}_{\mu}$, being an orbit of $\mb{H}_{\mu}$, also admits an action of a maximal torus $\mb{S}_{\mu}$ of $\mb{H}_{\mu}$.  Thus $\mb{Y}_{\mu}$ admits a class $[\mb{Y}_{\mu}]_{\mb{S}_{\mu}}$ in the $\mb{S}_{\mu}$-equivariant cohomology of $\mb{GL}_n/\mb{B}$, denoted by $H_{\mb{S}_{\mu}}^*(\mb{GL}_n/\mb{B})$.  In brief, this is a cohomology theory which is sensitive to the geometry of the $\mb{S}_{\mu}$-action on $\mb{GL}_n/\mb{B}$.  It admits a similar Borel-type presentation, this time as a polynomial ring in \textit{two} sets of variables (the usual set of $\xx$-variables referred to in the second paragraph, along with a second set which consists of $\yy$ and $\zz$-variables) modulo an ideal.  Moreover, the map $H_{\mb{S}_{\mu}}^*(\mb{GL}_n/\mb{B}) \rightarrow H^*(\mb{GL}_n/\mb{B})$ which sets all of the $\yy$ and $\zz$-variables to $0$ sends the equivariant class of any $\mb{S}_{\mu}$-invariant subvariety of $\mb{GL}_n/\mb{B}$ to its ordinary (non-equivariant) class.  Thus if a polynomial representative of $[\mb{Y}_{\mu}]_{\mb{S}_{\mu}}$ can be computed, one obtains a polynomial representative of $[\mb{Y}_{\mu}]$ by specializing $\yy,\zz \mapsto 0$.

In \cite{Wyser-13-TG}, this problem is solved for the case in which $\mu$ has only one part, in which case $\mb{H}_{\mu} = \mb{GO}_n$. Here, we extend the results of \cite{Wyser-13-TG} to give a formula for the equivariant class $[\mb{Y}_{\mu}]_{\mb{S}_{\mu}}$ (and $[\mb{Y}_{\mu}^\prime]_{\mb{S}_{\mu}^\prime}$) for an arbitrary composition $\mu$.  The main general result is Proposition \ref{prop:eqvt-cohomology-formula-general}; it, together with Proposition \ref{prop:formula-for-chern-class}, imply the case-specific equivariant formulas given in Corollaries \ref{cor:orthogonal-formula-eqvt} and \ref{cor:symplectic-formula-eqvt}.

The formulas for $[\mb{Y}_{\mu}]$ and $[\mb{Y}_{\mu}']$ obtained from these corollaries (by specializing $\yy$ and $\zz$-variables to $0$) are as follows:

\begin{corollary}\label{cor:orthogonal-formula-ordinary}
The ordinary cohomology class of $[\mb{Y}_{\mu}]$ is represented in $H^*(\mb{G}/\mb{B})$ by the formula
\[ 2^{d(\mu)}  \left( \displaystyle\prod_{i=1}^n x_i^{R(\mu,i) + \delta(\mu,i)}  \right) \displaystyle\prod_{i=1}^s \displaystyle\prod_{\nu_i+1 \leq j \leq k \leq \nu_{i+1} - j} (x_j+x_k). \]
\end{corollary}

\begin{corollary}\label{cor:symplectic-formula-ordinary}
The ordinary cohomology class of $[\mb{Y}_{\mu}']$ is represented in $H^*(\mb{G}/\mb{B})$ by the formula
\[ \left( \displaystyle\prod_{i=1}^{2n} x_i^{R(\mu,i)} \right) \displaystyle\prod_{i=1}^s  \displaystyle\prod_{\nu_i+1 \leq j < k \leq \nu_{i+1} - j} (\xx_j+\xx_k).�\]
\end{corollary}

The notations $\nu_i$, $d(\mu)$, $R(\mu,i)$, $\delta(\mu,i)$, etc. will be defined in Sections \ref{sec:background} and \ref{sec:eq-cohomology}.  For now, note that the representatives we obtain are factored completely into linear forms. In fact, the formulas reflect the semi-direct decomposition of $\mb{H}_{\mu}$ (resp., $\mb{H}_{\mu}'$) as we will detail in Section \ref{sec:eq-cohomology}.

A second possible way to approach the problem of computing $[\mb{Y}_{\mu}]$ is to write it as a non-negative integral linear combination of Schubert classes.  For each Weyl group element $w$ in the symmetric group $\mb{W} = S_n$, there is a \textbf{Schubert class} $[\mb{X}_w]$, the class of the \textbf{Schubert variety} $\mb{X}_w = \overline{\mb{B^+}w\mb{B}/\mb{B}}$ in $\mb{GL_n}/\mb{B}$, where $\mb{B^+}$ denotes the Borel subgroup of upper-triangular elements of $\mb{GL}_n$.  The Schubert classes form a $\mathbb{Z}$-basis for $H^*(\mb{GL}_n/\mb{B})$.

Assuming that one is able to compute the coefficients in
\begin{equation}\label{eq:schubert-sum}
[\mb{Y}_{\mu}] = \displaystyle\sum_{w \in \mb{W}} c_w [\mb{X}_w],
\end{equation}
then one may replace the Schubert classes in the above sum with the corresponding \textbf{Schubert polynomials} to obtain a polynomial in the $\xx$-variables representing $[\mb{Y}_{\mu}]$.  The Schubert polynomials $\frS_w$ are defined
recursively by first explicitly setting
$$
\frS_{w_0}:= x_1^{n-1} x_2^{n-2} \hdots x_{n-2}^2 x_{n-1},
$$
and then declaring that
$\frS_w = \partial_i \frS_{ws_i}$ if $ws_i < w$ in Bruhat order.  Here, $s_i = (i,i+1)$ represents the $i$th simple reflection, and $\partial_i$ represents the \textbf{divided difference operator} defined by
\[ \partial_i(f)(x_1,\hdots,x_n) = \dfrac{f(x_1,\hdots,x_n) - f(x_1,\hdots,x_{i-1},x_{i+1},x_i,x_{i+2},\hdots,x_n)}{x_i-x_{i+1}}. \]
It is well-known that $\frS_w$ represents $[\mb{X}_w]$ \cite{LS-82}, and so if \eqref{eq:schubert-sum} can be computed, $[\mb{Y}_{\mu}]$ is represented by the polynomial $\displaystyle\sum_{w \in \mb{W}} c_w \frS_w$.

In fact, a theorem due to M. Brion \cite{Brion-98} tells us in principle how to compute the sum \eqref{eq:schubert-sum} in terms of certain combinatorial objects. More precisely, to the variety $\mb{Y}_{\mu}$ there is an associated subset of $\mb{W}$, which we call the \textbf{$W$-set} of $\mb{Y}_{\mu}$, and denote by $W(\mb{Y}_{\mu})$.  For each $w \in W(\mb{Y}_{\mu})$, there is also an associated weight.  In fact, this weight is always a power of $2$, and the aforementioned theorem of Brion says that the sum \eqref{eq:schubert-sum} can be computed as
\begin{equation}\label{eq:brion-formula}
[\mb{Y}_{\mu}] = \displaystyle\sum_{w \in W(\mb{Y}_{\mu})} 2^{d(\mb{Y}_{\mu},w)} [\mb{X}_w]
\end{equation}

for non-negative integers $d(\mb{Y}_{\mu},w)$.  This can be turned into an explicit polynomial representative via the aforementioned Schubert polynomial recipe, assuming that one can compute the sets $W(\mb{Y}_{\mu})$, and the corresponding exponents $d(\mb{Y}_{\mu},w)$ explicitly.  In fact, in Section \ref{sec:bg-w-sets}, we recall explicit descriptions of the $W$-sets $W(\mb{Y}_{\mu})$ which have already been given in \cite{Can-Joyce-Wyser-15,Can-Joyce}, slightly extending those results to also give an explicit description of $W(\mb{Y}_{\mu}')$ for arbitrary $\mu$.  (Previous results of \cite{Can-Joyce-Wyser-15} only described $W(\mb{Y}_{\mu}')$ when $\mu$ consisted of a single part.)  And as we will note, the exponents $d(\mb{Y}_{\mu},w)$ are straightforward to compute.

This gives a second answer to our question, but note that it comes in a different form.  Indeed, the formulas of Corollaries \ref{cor:orthogonal-formula-ordinary} and \ref{cor:symplectic-formula-ordinary} are products of linear forms in the $\xx$-variables which are not obviously equal to the corresponding weighted sums of Schubert polynomials.  Of course, it is \emph{a priori} possible that the two polynomial representatives are actually \textit{not} equal, but simply differ by an element of $I^{\mb{W}}$, the ideal defining the Borel model of $H^*(\mb{G}/\mb{B})$.  However, our main result, Theorem \ref{thm:sum-equals-product}, states that in fact the apparent identity in $H^*(\mb{G}/\mb{B})$ is an equality \textit{of polynomials}.

The paper is organized as follows.  Section \ref{sec:background} is devoted mostly to recalling various background and preliminaries: We start by recalling necessary background on the wonderful compactification in Section \ref{sec:bg-wc}. We then give the explicit details of the examples which we are concerned with in Section \ref{sec:bg-examples}; this includes our conventions and notations regarding compositions, as well as our particular realizations of all groups, including the groups $\mb{H}_{\mu}$ and $\mb{H}_{\mu}'$.  In Section \ref{sec:bg-w-sets}, we review the notion of weak order and $W$-sets.  We recall results of \cite{Can-Joyce, Can-Joyce-Wyser-15} which are relevant to the current work, giving a slight extension of those results to the case of $W(\mb{Y}_{\mu}')$ for arbitrary $\mu$.

In Section \ref{sec:eq-cohomology}, we briefly review the necessary details of equivariant cohomology and the localization theorem.  We then use those facts to extend the formulas of \cite{Wyser-13-TG} to the more general cases of this paper, obtaining Proposition \ref{prop:eqvt-cohomology-formula-general} in a general setting, and its case-specific Corollaries \ref{cor:orthogonal-formula-eqvt} and \ref{cor:symplectic-formula-eqvt}. Corollaries \ref{cor:orthogonal-formula-ordinary} and \ref{cor:symplectic-formula-ordinary} are immediate consequences of these.

Finally, in Section \ref{sec:schubert-polys}, we compare the representatives of $[\mb{Y}_{\mu}]$ and $[\mb{Y}_{\mu}']$ obtained via our two different approaches, obtaining Theorem \ref{thm:sum-equals-product}.

\section{Background, notation, and conventions}\label{sec:background}

Throughout the text, we use italicized notation when we give arguments that apply to general reductive groups. In that case, $G$ is an arbitrary connected, reductive algebraic group defined over $\C$. We fix a Borel subgroup $B$ of $G$ and a maximal torus $T$ of $G$ contained in $B$. We let $W$ denote the Weyl group of $G$ and for every simple root $\alpha$ of $T$, we let $s_{\alpha} \in W$ denote the associated simple reflection and $P_{\alpha} = B \cup B s_{\alpha} B$ denote the minimal parabolic subgroup containing $B$ associated to $\alpha$.

By contrast, we use bolded notation to denote our two main examples $\mb{G} / \mb{H} = \mb{GL}_n / \mb{O}_n$ and $\mb{G} / \mb{H} = \mb{GL}_{2n} / \mb{Sp}_{2n}$. In these examples, we use the Borel subgroup $\mb{B}$ consisting of lower-triangular matrices and the maximal torus $\mb{T}$ consisting of diagonal matrices. Furthermore, the Weyl group $\mb{W}$ is isomorphic to the symmetric group $S_n$ (respectively, $S_{2n}$). We hope this helps the reader distinguish our case-specific results from the general results that we need along the way.

\subsection{The wonderful compactification}\label{sec:bg-wc}
We review the notion of the wonderful compactification of a general spherical homogeneous space. Using our convention outlined above, let $G$ be a connected, reductive algebraic group defined over $\C$. An algebraic subgroup $H$ of $G$, as well as the homogeneous space $G/H$, is called \textbf{spherical} if a Borel subgroup $B$ has finitely many orbits on $G/H$ (or equivalently, if $H$ has finitely many orbits on $G/B$).  Some such homogeneous spaces, namely partial flag varieties, are complete, while others (for example, symmetric homogeneous spaces) are not.  In the event that $G/H$ is not complete, a \textit{completion} of it is a complete $G$-variety $X$ which contains an open dense subset $X^0$ $G$-equivariantly isomorphic to $G/H$.  $X$ is a \textit{wonderful compactification} of $G/H$ if it is a completion of $G/H$ which is a ``wonderful'' spherical $G$-variety; this means that $X$ is a smooth spherical $G$-variety whose boundary (the complement of $X^0$) is a union of smooth, irreducible $G$-stable divisors $D_1,\hdots,D_r$ (the \textbf{boundary divisors}) with normal crossings and non-empty transverse intersections, such that the $G$-orbit closures on $X$ are precisely the partial intersections of the $D_i$'s.  The number $r$ is called the \textit{rank} of the homogeneous space $G/H$.

The number of $G$-orbits on $X$ is then $2^r$, and they are parametrized by subsets of $\{1,\hdots,r\}$, with a given subset determining the orbit by specifying the set of boundary divisors containing its closure.  It is well-known that the subsets of $\{1,2,\hdots,r\}$ are in bijection with the compositions of $r + 1$.  A \textbf{composition} of $n$ is simply a tuple $\mu=(\mu_1,\hdots,\mu_s)$ with $\sum_i \mu_i = n$. For a given composition $\mu=(\mu_1,\hdots,\mu_s)$, we define the integers $\nu_2,\hdots,\nu_s$  by the formula $\nu_i = \sum_{j=1}^{i-1} \mu_j$ for $i=2,\dots, s$. By convention, we set $\nu_1 = 0$. In words, $\nu_i$ is the sum of the first $i-1$ parts of the composition $\mu$.
We parametrize the $\mb{G}$-orbits ($\mb{G} = \mb{GL}_n$ or $\mb{G} = \mb{GL}_{2n}$) in our examples by compositions of $n$, with $n-1$ being the rank of both of the symmetric spaces $\mb{GL}_n/\mb{GO}_n$ and $\mb{GL}_{2n}/\mb{GSp}_{2n}$.


\subsection{Our examples}\label{sec:bg-examples}
We now describe the two primary examples to which we will directly apply the general results of this paper.  The first is the wonderful compactification $\mb{X}$ of the space of all smooth quadric hypersurfaces in $\mathbb{P}^{n-1}$, i.e. $G/H$, where $(\mb{G},\mb{H})=(\mb{GL}_n,\mb{GO}_n)$, classically known as the variety of complete quadrics.

Choose $\mb{B}$ to be the lower-triangular subgroup of $\mb{GL}_n$, and $\mb{T}$ to be the maximal torus consisting of diagonal matrices.  We realize $\mb{H}_0=\mb{O}_n$ as the fixed points of the involution given by $\theta(g) = J (g^t)^{-1} J$, where $J$ is the $n \times n$ matrix with $1$'s on the antidiagonal, and $0$'s elsewhere.  When $\mb{H}_0$ is realized in this way, $\mb{H}_0 \cap \mb{B}$, the lower-triangular subgroup of $\mb{H}_0$, is a Borel subgroup, and $\mb{S}_0 := \mb{H}_0 \cap \mb{T}$ is a maximal torus of $\mb{H}_0$, consisting of all elements of the form
\begin{equation}\label{eq:orthogonal-torus-even}
\text{diag}(a_1,\hdots,a_m,a_m^{-1},\hdots,a_1^{-1}),
\end{equation}
where $a_i \in \C^*$ for $i=1,\dots,m$ when $n=2m$ is even, and of the form
\[ \text{diag}(a_1,\hdots,a_m,1,a_m^{-1},\hdots,a_1^{-1}), \]
where $a_i \in \C^*$ for $i=1,\dots,m$ when $n=2m+1$ is odd.
The Lie algebra $\mb{\mathfrak{s}}_0$ of $\mb{S}_0$ then takes the form
\begin{equation}\label{eq:orthogonal-torus-lie-even}
\text{diag}(a_1,\hdots,a_m,-a_m,\hdots,-a_1),
\end{equation}
where $a_i \in \C$ for $i=1,\dots,m$ in the even case, and
\[ \text{diag}(a_1,\hdots,a_m,0,-a_m,\hdots,-a_1), \]
where $a_i \in \C$ for $i=1,\dots,m$ in the odd case.

Note that the diagonal elements of $H$ form a maximal torus $\mb{S}$ of dimension one greater than $\dim \mb{S_0}$.  The general element of $\mb{S}$ is of the form
\begin{equation}\label{eq:orthogonal-bigger-torus-even}
\text{diag}(\lambda a_1,\hdots,\lambda a_m,\lambda a_m^{-1},\hdots,\lambda a_1^{-1})
\end{equation}
in the even case, and of the form
\begin{equation}\label{eq:orthogonal-bigger-torus-odd}
\text{diag}(\lambda a_1,\hdots,\lambda a_m,\lambda,\lambda a_m^{-1},\hdots,\lambda a_1^{-1})
\end{equation}
in the odd case. Here, $\lambda$ is an element of $\C^*$ and the $a_i$'s are as before.

The Lie algebra $\mb{\mathfrak{s}}$ of $\mb{S}$ then consists of diagonal matrices of the form
\begin{equation}\label{eq:orthogonal-bigger-torus-lie-even}
\text{diag}(\lambda + a_1,\hdots,\lambda + a_m,\lambda - a_m,\hdots,\lambda - a_1)
\end{equation}
in the even case, and of the form
\begin{equation}\label{eq:orthogonal-bigger-torus-lie-odd}
\text{diag}(\lambda + a_1,\hdots,\lambda + a_m,\lambda,\lambda - a_m,\hdots,\lambda - a_1).
\end{equation}
in the odd case. Here, $\lambda$ is an element of $\C$ and the $a_i$'s are as before.

Thus we have described the homogeneous space $\mb{G}/\mb{H}$, where $\mb{G} = \mb{GL}_n$ and $\mb{H} = \mb{GO}_n$, which is the dense $\mb{G}$-orbit on $\mb{X}$. We now describe the other $\mb{G}$-orbits.  As mentioned in Section \ref{sec:bg-wc}, they are in bijection with compositions $\mu$ of $n$.

Corresponding to $\mu$, we have a standard parabolic subgroup $\mb{P}_{\mu}=\mb{L}_{\mu} \ltimes \mb{U}_{\mu}$ containing $\mb{B}$ whose Levi factor $\mb{L}_{\mu}$ is $\mb{GL}_{\mu_1} \times \hdots \times \mb{GL}_{\mu_s}$, embedded in $\mb{GL}_n$ in the usual way, as block diagonal matrices. The $\mb{G}$-orbit $\caO_{\mu}$ corresponding to $\mu$ is then isomorphic to $\mb{G}/\mb{H}_{\mu}$, where $\mb{H}_{\mu}$ is the group
\[ (\mb{GO}_{\mu_1} \times \hdots \times \mb{GO}_{\mu_s}) \ltimes \mb{U}_{\mu}, \]
where $\mb{GO}_{\mu_i} = \mb{Z}_{\mu_i} \mb{O}_{\mu_i}$ is realized in $\mb{GL}_{\mu_i}$ as described above. Then $\mb{B} \cap \mb{H}_{\mu}$ is a Borel subgroup of $\mb{H}_{\mu}$, and $\mb{S}_{\mu} := \mb{T} \cap \mb{H}_{\mu}$ is a maximal torus of $\mb{H}_{\mu}$.

Note that $\mb{S}_{\mu}$ is diagonal, and consists of $s$ ``blocks'', the $i$th block consisting of those diagonal entries in the range $\nu_i+1,\hdots,\nu_i+\mu_i$.
If $\mu_i=2m$ is even, then the $i$th block is of the form
\begin{equation}\label{eq:block-torus-even}
\text{diag}(\lambda_i a_{i,1},\hdots,\lambda_i a_{i,m},\lambda_i a_{i,m}^{-1},\hdots,\lambda_i a_{i,1}^{-1}).
\end{equation}
The corresponding $i$th block of an element of ${\mb{\mathfrak{s}}}_{\mu}$ is then of the form
\begin{equation}\label{eq:block-torus-lie-even}
\text{diag}(\lambda_i + a_{i,1},\hdots,\lambda_i + a_{i,m},\lambda_i - a_{i,m},\hdots,\lambda_i - a_{i,1}).
\end{equation}

If $\mu_i=2m+1$ is odd, then the $i$th block is of the form
\begin{equation}\label{eq:block-torus-odd}
\text{diag}(\lambda_i a_{i,1},\hdots,\lambda_i a_{i,m},\lambda_i, \lambda_i a_{i,m}^{-1},\hdots,\lambda_i a_{i,1}^{-1}).
\end{equation}
The $i$th block of an element of $\mb{\mathfrak{s}}_{\mu}$ is correspondingly of the form
\begin{equation}\label{eq:block-torus-lie-odd}
\text{diag}(\lambda_i + a_{i,1},\hdots,\lambda_i + a_{i,m},\lambda_i,\lambda_i-a_{i,m},\hdots,\lambda_i-a_{i,1}).
\end{equation}

Our second primary example is the wonderful compactification of $(\mb{G},\mb{H}^\prime)$, $(\mb{G},\mb{H}^\prime) = (\mb{GL}_{2n},\mb{GSp}_{2n}$.  $\mb{H}^\prime$ is a central extension of $\mb{H}_0^\prime=\mb{Sp}_{2n}$, the latter group being realized as the fixed points of the involutory automorphism of $\mb{GL}_{2n}$ given by $g \mapsto \widetilde{J} (g^t)^{-1} \widetilde{J}$, where $\widetilde{J}$ is the $2n \times 2n$ antidiagonal matrix whose antidiagonal consists of $n$ $1$'s followed by $n$ $-1$'s, reading from the northeast corner to the southwest. Let $\mb{S}_0' := \mb{Sp}_{2n} \cap \mb{T}$.

Once again taking $\mb{B}$ to be the lower-triangular Borel of $\mb{GL}_{2n}$, and $\mb{T}$ to be the diagonal maximal torus of $\mb{GL}_{2n}$, one checks that $\mb{H}_0^\prime \cap \mb{B}$ is a Borel subgroup of $\mb{H}_0^\prime$, and that $\mb{S}_0^\prime := \mb{H}_0^\prime \cap \mb{T}$ is a maximal torus of $\mb{H}_0^\prime$. The corresponding torus $\mb{S}^\prime$ of $\mb{H}^\prime$ is then of exactly the same format as indicated in \eqref{eq:orthogonal-bigger-torus-even}, while its Lie algebra $\mb{\mathfrak{s}}'$ is as indicated by \eqref{eq:orthogonal-bigger-torus-lie-even}.

The additional $\mb{G}$-orbits on the wonderful compactification $\mb{X}^\prime$ of $\mb{GL}_{2n}/\mb{GSp}_{2n}$ again correspond to compositions $\mu=(\mu_1,\hdots,\mu_s)$ of $n$.
For such a composition, we let $\mb{P}_{\mu}=\mb{L}_{\mu} \ltimes \mb{U}_{\mu}$ be the standard parabolic subgroup whose Levi factor is $\mb{GL}_{2\mu_1} \times \hdots \times \mb{GL}_{2\mu_s}$,
embedded in $\mb{GL}_{2n}$ as block diagonal matrices.  Then the $\mb{G}$-orbit corresponding to $\mu$ is isomorphic to $\mb{G}/\mb{H}_{\mu}'$, where
\[ \mb{H}_{\mu}' = (\mb{GSp}_{2\mu_1} \times \hdots \times \mb{GSp}_{2\mu_s}) \ltimes \mb{U}_{\mu}, \]
with each $\mb{GSp}_{2\mu_i} = \mb{Z}_{2\mu_i} \mb{Sp}_{2\mu_i}$ embedded in the corresponding $\mb{GL}_{2\mu_i}$ just as described above.

The torus $\mb{S}_{\mu}'$ then consists of $s$ ``blocks'', just as in the orthogonal case.  This time, each
block is of even dimension, so each is of the form described by \eqref{eq:block-torus-even}.  The corresponding block
of the Lie algebra $\mb{\mathfrak{s}}_{\mu}'$ is then of the form indicated in \eqref{eq:block-torus-lie-even}.

\subsection{Some general results}\label{sec:bg-general}

We now introduce several general observations that are applicable to our chosen examples.  In this subsection, returning to the conventions set forth in the beginning of Section \ref{sec:background}, let $G$ be an arbitrary connected, reductive algebraic group over $\mathbb{C}$ and let $P$ be a parabolic subgroup of $G$ containing the Borel subgroup $B$ with Levi decomposition $P = L \ltimes U$, where $L$ is a Levi subgroup of $G$ containing $T$ and $U$ is the unipotent radical of $P$.  Let $H_L$ be a subgroup of $L$ and consider the $G$-variety $V = G \times^P L / H_L$, the quotient of $G \times L / H_L$ by the action of $P$, where $P$ acts on $G$ by right multiplication and on $L / H_L$ via its projection to $L$. Note that $V$ is a $G$-variety via left multiplication of $G$ on the first factor.$V$ is a homogeneous $G$-variety, and the stabilizer subgroup of the point $[1, 1 H_L / H_L]$ is $H := H_L U \subseteq P$.  The construction of $V \cong G/H$ from $L/H_L$ (or equivalently of $H \subseteq G$ from $H_L \subseteq L$) is called \emph{parabolic induction}.

Note that in our examples, $\mb{H}_{\mu}$ is obtained via parabolic induction from $\mb{GO}_{\mu_1} \times \cdots \times \mb{GO}_{\mu_k}$ (playing the role of $H_L$) and $\mb{H}_{\mu}'$ is obtained via parabolic induction from $\mb{GSp}_{\mu_1} \times \cdots \times \mb{GSp}_{\mu_k}$ (likewise playing the role of $H_L$).  We now summarize some results from the literature that describe the role of parabolic induction for wonderful varieties.

\begin{definition}
Let $G$ be a reductive algebraic group. A subgroup $H$ of $G$ is said to be \textbf{symmetric} if there exists an algebraic involution $\theta : G \rightarrow G$ such that, letting $K = G^{\theta} = \{ g \in G : \theta(g) = g \}$ and letting $Z$ denote the center of $G$, we have $ZK^0 \subseteq H \subseteq ZK$.
\end{definition}

\begin{proposition}\label{prop:par ind}
\begin{enumerate}
\item Let $H$ be a symmetric subgroup of $G$ such that $G/H$ has a wonderful compactification $X$.
Then every $G$-orbit of $X$ is obtained via parabolic induction from some symmetric homogeneous space $L/H_L$ associated to some Levi subgroup $L$ of $G$.
\item If $H_L$ is a spherical subgroup of $L$ such that $L / H_L$ contains a single closed $B_L$-orbit (with $B_L$ a Borel subgroup of $L$) and $H$ is the subgroup of $G$ obtained by parabolic induction, i.e. $G/H \cong G \times^P L/H_L$, then $H$ is a spherical subgroup of $G$ and $G/H$ contains a single closed $B$-orbit.
\end{enumerate}
\end{proposition}

\begin{proof}
The first result is a reformulation of a result of de Concini and Procesi \cite[Theorem 5.2]{DeConcini-Procesi}.  They show that a $G$-orbit $V$ has a $G$-equivariant map $V \rightarrow G / P$ with fiber $L / H_L$.  (In fact, they show that the closure of $V$ maps $G$-equivariantly to $G /P$ with fiber the wonderful compactifcation of $L / H_L$, from which our statement follows by restricting the map to $V$.)  It follows that there is a bijective morphism $\phi: G \times^P L / H_L \rightarrow V$, which is an isomorphism if $\phi$ is separable \cite[discussion after Theorem 2.2]{Timashev}.

The second result follows from \cite[Lemma 6]{Brion-01}, but we give a direct proof for completeness. We again consider the $G$-equivariant fibre bundle $\pi: G/H \rightarrow G/P$, with fiber $L/H_L$. Since $H_L$ is a spherical subgroup of $L$ and $P$ is a spherical subgroup of $G$ (by the classical Bruhat decomposition), it follows that $H$ is a spherical subgroup of $G$. Moreover, since there is a unique closed $B_L$-orbit in $L / H_L$ and a unique closed $B$-orbit in $G/P$, there is a unique closed $B$-orbit in $G/H$, namely the $B$-orbit which maps via $\pi$ to the closed $B$-orbit of $G/P$ and whose fiber over the base point $1P/P$ is identified with the closed $B_L$-orbit in $L / H_L$.

\end{proof}

\subsection{Weak order and $W$-sets}\label{sec:bg-w-sets}

Let $H$ be a spherical subgroup of the connected, reductive algebraic group $G$ and assume that there exists a wonderful compactification $X$ of $G/H$. We continue with the notation of the previous subsection. In this subsection we review the notion of the weak order on the set of $B$-orbit closures of $X$. Note that $X$ is a spherical variety, meaning that $B$ has finitely many orbits on $X$, so the set of $B$-orbits equipped with the weak order is a finite poset.

The \textbf{weak order} on the set of $B$-orbit closures of $X$ is the one whose covering relations are given by $Y \prec Y'$ if and only if $Y' = P_{\ga} Y \neq Y$ for some simple root $\ga$ of $T$ relative to $B$. In general, $Y \leq Y'$ if and only if $Y'=P_{\ga_s} \hdots P_{\ga_1}Y$ for some sequence of simple roots $\ga_1,\hdots,\ga_s$.

When considering the weak order on $X$, it suffices to consider it on the individual $G$-orbits separately. Indeed, if $Y$ and $Y'$ are the closures of $B$-orbits $Q$ and $Q'$, respectively, and if $Y \leq Y'$ in weak order, then $Q$ and $Q'$ lie in the same $G$-orbit.  Therefore, we focus on the weak order on $B$-orbit closures on a homogeneous space $G/H$. The Hasse diagram of the weak order poset can be drawn as a graph with labeled edges,
each edge with a weight of  either $1$ or $2$. This is done as follows: For each cover $Y \prec Y'$ with $Y'=P_{\ga}Y$, we draw an edge from $Y$ to $Y'$, and label it by the simple reflection $s_{\ga}$.  If the natural map $P_{\ga} \times^B Y \rightarrow Y'$ is birational, then the edge has weight $1$; if the map is generically $2$-to-$1$, then the edge has weight $2$.  (These are the only two possibilities.)  The edges of weight $2$ are frequently depicted as double edges \cite{Brion-01}.

In the graph described above, there is a unique maximal element, since $G/H$ is the closure of its dense $B$-orbit. Given a $B$-orbit closure $Y$, its \textbf{$\mb{W}$-set}, denoted $W(Y)$, is defined as the set of all elements of $W$ obtained by taking the product of edge labels of paths which start at $Y$ and end at $G/H$. The weight $d(Y,w)$ alluded to before \eqref{eq:brion-formula} is defined as the number of double edges in any such path whose edge labels multiply to $w$. (Note that there is one such path for each reduced expression of $w$, but all such paths have the same number of double edges, so that $d(Y,w)$ \textit{is} well-defined \cite{Brion-01}.)  We have now recalled all explanation necessary to understand \eqref{eq:brion-formula}.

Next, we briefly recall results of \cite{Can-Joyce,Can-Joyce-Wyser-15} which give explicit descriptions of these $W$-sets in the cases described in Section \ref{sec:bg-examples}.

We begin by addressing the case of the extended orthogonal group $\mb{H} = \mb{GO}_n$ and the variants $\mb{H}_{\mu}$.  For a set $A \subseteq [n] := \{1, 2,\dots, n \}$, say that $a < b$ are \emph{adjacent in $A$} if there does not exist $c \in A$ such that $a < c < b$.  Let $\mathcal{W}_n$ denote the set of permutations $w \in S_n$ that have the following recursive property.  Initialize $A_1 = [n]$.  For $1 \leq i \leq \lfloor n/2 \rfloor$, assume that $w(1), \dots, w(i-1)$ and $w(n+2-i), \dots, w(n)$ have already been defined.  (This condition is vacuous in the case $i = 1$.)  Then $w(i)$ and $w(n+1-i)$ must be adjacent in $A_i$ and $w(i)$ must be greater than $w(n+1-i)$.  Define $A_{i+1} := A_i \setminus \{ w(i), w(n+1-i) \}$.  This completely defines $w$ when $n$ is even, and if $n = 2k+1$ is odd, then $A_{k+1}$ will consist of a single element $m$, so define $w(k+1) = m$.  For example, $\mathcal{W}_5$ consists of the eight elements of $S_5$ given in one-line notation by $24531$, $25341$, $34512$, $35142$, $42513$, $45123$, $52314$, $53124$.

\begin{proposition}[\cite{Can-Joyce}]\label{prop:w-set-orthogonal-dense}
Let $\mb{Y}$ denote the closed $\mb{B}$-orbit in $\mb{G}/\mb{H}$ where $\mb{G} = \mb{GL}_n$ and $\mb{H} = \mb{GO}_n$.  Then $W(\mb{Y}) = \mathcal{W}_n$.
\end{proposition}



\begin{remark}\label{rem:w-set-comparison}
The ``$\mathcal{W}$-set'' of~\cite{Can-Joyce}, which is denoted by $D_n$ there, differs slightly from ours.
More precisely, the relationship between $D_n$ and our $\mathcal{W}$-set is
$$
\mathcal{W}_n = \{ w_0 w^{-1} w_0 : w \in D_n \}.
$$
Let us explain the reason for the discrepancy. First, the partial order considered in \cite{Can-Joyce} is the opposite of the weak order on Borel orbits considered here, which necessitates inverting the elements of $D_n$. Second, we consider here $\mb{B}$ to be the Borel subgroup of lower triangular matrices in $GL_n$, while \cite{Can-Joyce} uses the Borel subgroup of upper triangular matrices. This necessitates conjugating the elements by $w_0$.
\end{remark}

Similarly, we define a set $\mathcal{W}_{\mu} \subseteq S_n$ associated with a composition $\mu = (\mu_1, \dots, \mu_s)$ of $n$. We begin by recalling the notion of a $\mu$-string \cite{Can-Joyce}. Recall that we have defined $\nu_k = \sum_{j=0}^{k-1} \mu_j$ for $2 \leq k \leq s$; by convention $\nu_1 = 0$. The $i$th $\mu$-string of a permutation $w \in S_n$, denoted by $\text{str}_i(w)$ is the word $w(\nu_i + 1)\, w(\nu_i + 2)\, \dots w(\nu_{i+1})$. For example, if $w = 3715462$ is a permutation from $S_7$ (written in one-line notation) and $\mu = (2,4,1)$, then the second $\mu$-string of $w$ is the word $1546$. Let $A \subseteq [n]$ have cardinality $k$ and assume a word $\omega$ of length $k$ is given that uses each letter of $A$ exactly once. Define a bijection between $[k]$ and $A$ by associating to $i \in [k]$ the $i$th largest element of $A$. Under this bijection, the word $\omega$ corresponds to the one-line notation of a permutation $w$ in $S_k$. Call $w$ the permutation associated to the word $\omega$.  Continuing the example, the permutation associated to the word $\omega = 1546$ is $1324 \in S_4$ (in one-line notation).

The set $\mathcal{W}_{\mu}$ consists of all $w \in S_n$ such that the letters of $\text{str}_i(w)$ are precisely those $j$ such that $ n - \nu_{i+1} < j \leq n - \nu_i $ and the permutation associated to $\text{str}_i(w)$ is an element of $\mathcal{W}_{\mu_i}$.  For example, $\mathcal{W}_{(4,2)}$ consists of the three elements of $S_6$ given in one-line notation by $465321$, $563421$, $643521$.

\begin{proposition}[\cite{Can-Joyce}]\label{prop:w-set-orthogonal-general}
Let $\mb{Y}_{\mu}$ denote the closed $\mb{B}$-orbit in $\mb{G}/\mb{H}_{\mu}$ where $\mb{G} = \mb{GL}_n$ and $\mb{H}_{\mu} = (\mb{GO}_{\mu_1} \times \cdots \times \mb{GO}_{\mu_s}) \ltimes \mb{U}_{\mu}$, i.e. $\mb{H}_{\mu}$ is obtained by parabolic induction from $\mb{H}_L = \mb{GO}_{\mu_1} \times \cdots \times \mb{GO}_{\mu_s} \subseteq \mb{L} = \mb{GL}_{\mu_1} \times \cdots \times \mb{GL}_{\mu_s}$ to $\mb{G} = \mb{GL}_n$.  Then $W(\mb{Y}_{\mu}) = \mathcal{W}_{\mu}$.
\end{proposition}


Just as in Remark \ref{rem:w-set-comparison}, the relation between $\mathcal{W}_{\mu}$ and the set $D_{\mu}$ defined in \cite{Can-Joyce} is $\mathcal{W}_{\mu} = \{ w_0 w^{-1} w_0 : w \in D_{\mu} \}$.

We now turn to the extended symplectic case $\mb{H'} = \mb{GSp}_{2n}$ and its variants $\mb{H}_{\mu}'$. Consider the inclusion of $S_n$ into $S_{2n}$ via the map $u \mapsto \phi(u)=v=v_1v_2\dots v_{2n}$, where
$$
[v_1,v_2,\dots,v_n, v_{n+1},\dots,v_{2n-1}, v_{2n} ]  = [2u(1) - 1, 2u(2) - 1, \dots, 2u(n) -1, 2u(n), \dots, 2u(2), 2u(1)].
$$
Let $\mathcal{W}'_{2n} = \{ \phi(u) \in S_{2n} : u \in S_n \}$.
For example, $\mathcal{W}'_6$ consists of the six elements of $S_6$ given in one-line notation by $135642$, $153462$, $315624$, $351264$, $513426$, $531246$.

\begin{proposition}[\cite{Ressayre10, Can-Joyce-Wyser-15}]\label{prop:w-set-symplectic-dense}
Let $\mb{Y'}$ denote the closed $\mb{B}$-orbit in $\mb{G}/\mb{H'}$ where $\mb{G} = \mb{GL}_{2n}$ and $\mb{H'} = \mb{GSp}_{2n}$.
Then $W(\mb{Y'}) = \mathcal{W}'_n$.
\end{proposition}

We now proceed to define a set $\mathcal{W}'_{\mu} \subseteq S_{2n}$ for any composition $\mu = (\mu_1, \dots, \mu_s)$ of $n$.  The set $\mathcal{W}'_{\mu}$ consists of all $w \in S_n$ such that the letters of $\text{str}_i(w)$ are precisely those $j$ such that $n - \nu_{i+1} < j \leq n - \nu_i$ and the permutation associated to $\text{str}_i(w)$ is an element of $\mathcal{W}'_{2\mu_i}$.  For example, $\mathcal{W}'_{(2,4)}$ consists of the two elements of $S_{6}$ given in one-line notation by $123564$, $125346$.

\begin{proposition}\label{prop:w-set-symplectic-general}
Let $\mb{Y}_{\mu}'$ denote the closed $\mb{B}$-orbit in $\mb{G}/\mb{H}_{\mu}'$ where $\mb{G} = \mb{GL}_{2n}$ and $\mb{H}_{\mu}' = (\mb{GSp}_{2\mu_1} \times \cdots \times \mb{GSp}_{2\mu_s}) \ltimes \mb{U}_{\mu}$, i.e. $\mb{H}_{\mu}'$ is obtained by parabolic induction from $\mb{H}_L = \mb{GSp}_{2\mu_1} \times \cdots \times \mb{GSp}_{2\mu_s} \subseteq \mb{L} = \mb{GL}_{2\mu_1} \times \cdots \times \mb{GL}_{2\mu_s}$ to $\mb{G} = \mb{GL}_{2n}$.  Then $W(\mb{Y}_{\mu}') = \mathcal{W}'_{\mu}$.
\end{proposition}

Proposition \ref{prop:w-set-symplectic-general} is proved in exactly the same manner as Proposition \ref{prop:w-set-orthogonal-general} is proven in \cite[Theorem 4.11]{Can-Joyce}, so we omit its proof.  Alternatively, it can be obtained as a corollary of Proposition \ref{prop:w-set-symplectic-dense} by applying a general result of Brion on $W$-sets for homogeneous spaces obtained by parabolic induction \cite[Lemma 1.2]{Brion-98}.


\section{Equivariant cohomology computations}\label{sec:eq-cohomology}
\subsection{Background}\label{sec:eqvt-cohomology-background}
We start by reviewing the basic facts of equivariant cohomology that we will need to support our method of computation.  All cohomology rings use $\mathbb{Q}$-coefficients. Results of this section are generally stated without proof, as they are fairly standard.  To the reader seeking a reference we recommend \cite{Wyser-13-TG} for an expository treatment, as well as references therein.

We will apply our results to equivariant cohomology with respect to the action of $\mb{S}_{\mu}$ (respectively, $\mb{S}_{\mu}'$) on $\mb{G}/\mb{B}$, these tori having been defined in Section~\ref{sec:bg-examples}. Given a variety $X$ with an action of an algebraic torus $S$ with Lie algebra $\frs$, the equivariant cohomology is, by definition,
\[ H_S^*(X) := H^*((ES \times X)/S), \]
where $ES$ denotes a contractible space with a free $S$-action.  $H_S^*(X)$ is an algebra for the ring $\Lambda_S := H_S^*(\{\text{pt.}\})$, the $\Lambda_S$-action being given by pullback through the obvious map $X \rightarrow \{\text{pt.}\}$. The ring $\Lambda_S$ is naturally isomorphic to the symmetric algebra $\text{Sym}(\frs^*)$ on $\frs^*$. Thus, if $y_1,\hdots,y_n$ are a basis for $\frs^*$, then $\Lambda_S\simeq \text{Sym}(\frs^*)$ is isomorphic to the polynomial ring $\Q[\yy] = \Q[y_1,\hdots,y_n]$. When $X=G/B$ with $G$ a reductive algebraic group and $B$ a Borel subgroup, and if $S \subseteq T \subseteq B$ with $T$ a maximal torus in $G$, then we have the following concrete description of $H_S^*(X)$:

\begin{prop}\label{prop:eqvt-cohom-flag-var}
Let $R = \text{Sym}(\frt^*)$, $R' = \text{Sym}(\frs^*)$.  Then $H_S^*(X) = R' \otimes_{R^W} R$. If $X_1,\hdots,X_n$ are a basis for $\frt^*$, and $Y_1,\hdots,Y_m$ are a basis of $\frs^*$, elements of $H_S^*(X)$ are thus represented by polynomials in variables $x_i := 1 \otimes X_i$ and $y_i := Y_i \otimes 1$.
\end{prop}

To make this clear in the setting of our examples (cf. Section \ref{sec:bg-examples}), if $\mb{S}$ is taken to be the full maximal torus $\mb{T}$ of $\mb{GL}_n$, we let $\mb{X_i}$ ($i=1,\hdots,n$) be the function on $\frt$ which evaluates to $a_i$ on the element \[ t= \text{diag}(a_1,\hdots,a_n). \] We denote by $\mb{Y_i}$ ($i=1,\hdots,n$) a second copy of the same set of functions.  We then have two sets of variables as in Proposition~\ref{prop:eqvt-cohom-flag-var}, typically denoted $\xx = x_1,\hdots,x_n$ and $\yy = y_1,\hdots,y_n$, and $\mb{T}$-equivariant classes are represented by polynomials in these variables.

If $\mu=(\mu_1,\hdots,\mu_s)$ is a composition of $n$, let $\mb{T}$ be the full torus of $\mb{GL}_n$, and let $\mb{S}_{\mu}$ be the torus of $\mb{H}_{\mu}$, as in Section \ref{sec:bg-examples}. We denote by $\mb{X}_i$ the same function on $\frt$ as described above. We denote by $\mb{Y}_{i,j}$ the function on $\frs_{\mu}$ which evaluates to $a_{i,j}$ on an element of the form in \eqref{eq:block-torus-lie-even} (if $\mu_i$ is even) or \eqref{eq:block-torus-lie-odd} (if $\mu_i$ is odd).  We denote by $\mb{Z}_i$ the function which evaluates to $\lambda_i$ on an element of the form \eqref{eq:block-torus-lie-even} or \eqref{eq:block-torus-lie-odd}.  Then letting lower-case $\xx$, $\yy$, and $\zz$-variables correspond to these coordinates (with matching indices), $H_{\mb{S}_{\mu}}^*(\mb{G}/\mb{B})$ is generated by these variables, and when we seek formulas for certain $\mb{S}_{\mu}$-equivariant classes, we are looking for polynomials in these particular variables.

We next recall the standard localization theorem for torus actions. For more on this fundamental result, the reader may consult, for example, \cite{Brion-98_i}.

\begin{theorem}\label{thm:eqvt-localization}
Let $X$ be an $S$-variety, and let $i: X^S \hookrightarrow X$ be the inclusion of the $S$-fixed locus of $X$.
Then the pullback map of $\Lambda_S$-modules
\[ i^*: H_S^*(X) \rightarrow H_S^*(X^S) \]
is an isomorphism after a localization which inverts finitely many characters of $S$.
In particular, if $H_S^*(X)$ is free over $\Lambda_S$, then $i^*$ is injective.
\end{theorem}

When $X=G/B$, $H_S^*(X) = R' \otimes_{R^W} R$ is free over $R'$ (as $R$ is free over $R^W$), so any equivariant class is entirely determined by its image under $i^*$. We will only apply this result in the event that $S=T$ is the full maximal torus of $G$, so the $S$-fixed locus is finite, being parametrized by $W$.  Then for us,
\[ H_S^*(X^S) \cong \bigoplus_{w \in W} \Lambda_S, \]
so that in fact a class in $H_S^*(X)$ is determined by its image under $i_w^*$ for each $w \in W$, where here $i_w$ denotes the inclusion of the $S$-fixed point $wB/B$ in $G/B$.  Given a class $\beta \in H_S^*(X)$ and an $S$-fixed point $wB/B$, we may denote the restriction $i_w^*(\beta)$ at $wB/B$ by $\beta|_w$.

We end the section by recalling how the restriction maps are computed.
\begin{prop}\cite{Wyser-13-TG}\label{prop:restriction-maps}
With the notation of the preceding paragraph, suppose that $\beta \in H_S^*(X)$ is represented by the polynomial $f = f(\xx,\yy)$ in variables $\xx$, $\yy$.  Then $\beta|_w \in \Lambda_S$ is the polynomial $f(wY|_{\mathfrak{s}},Y)$.
\end{prop}

\subsection{A general result}

We continue with the general setup given at the beginning of Section \ref{sec:background}. Let $P = LU$ be a parabolic subgroup of $G$ containing $B$, with $L$ a Levi factor containing $T$ and $U$ the unipotent radical of $P$ contained in $B$. Let $H_L$ denote a spherical subgroup of $L$ and let $S := T \cap H_L$ be a maximal torus of $H_L$. Let $X$ denote the generalized flag variety $G / B$.

Let $H$ be the subgroup obtained by parabolic induction from $H_L \subseteq L$ to $G$, i.e. $H = H_L U$. As proven in Proposition \ref{prop:par ind}, $H$ is a spherical subgroup of $G$ with a unique closed $B$-orbit $Q$. The torus $S$ is a maximal torus of $H$, and we seek to describe $[Q] \in H^*_S(X)$.


Denote by $B_L$ the Borel subgroup $B \cap L$ of $L$, and let $Y$ denote the generalized flag variety $L/B_L$. Note that there is an $S$-equivariant embedding $j: Y \hookrightarrow X$, which induces a pushforward map in cohomology $j_*: H_S^*(Y) \rightarrow H_S^*(X)$.  Given our setup, the orbit $Q' = H \cdot 1B_L/B_L$ is closed in $Y$.  In fact, $j(Q') = Q$. To see this we just observe that $1B/B \cong 1B_L/B_L$ under the embedding $j$. Nonetheless, to avoid possible confusion, we refer to the closed orbit as $Q'$ when thinking of it as a subvariety of $Y$, and as $Q$ when thinking of it as a subvariety of $X$.

Let $W_L \subseteq W$ denote the Weyl group of $L$.  In the root system $\Phi$ for $(G,T)$, choose $\Phi^+$ to be the positive system such that the roots of $B$ are negative.  Similarly, let $\Phi_L$ be the root system for $(L,T)$ and let $\Phi_L^+ = \Phi_L \cap \Phi^+$ be the positive roots of $L$ such that the roots of $B_L$ are negative.  Given a root $r \in \Phi$ (resp., $r \in \Phi_L$), let $\mathfrak{g}_r$ denote the associated one-dimensional root space in $\mathfrak{g}$ (resp., $\mathfrak{l}$).

The next result relates the class of $Q'$ in $H_S^*(Y)$ to the class of $Q$ in $H_S^*(X)$.

\begin{prop}\label{prop:eqvt-cohomology-formula-general}
With notation as above, the classes $j^*[Q]$ and $[Q']$ are related via multiplication by the top $S$-equivariant Chern class of the normal bundle $N_Y X$, i.e. $j^*[Q] = c_d^S(N_Y X) \cap [Q']$.  This Chern class is the restriction of a $T$-equivariant Chern class for the same normal bundle, and the latter class, which we denote by $\ga$, is uniquely determined by the following properties:
  \[ \ga|_w = \begin{cases}
               \displaystyle\prod_{r \in \Phi^+ \setminus \Phi_L^+} wr & \text{ if $w \in W_L$}, \\
               0 & \text{ otherwise.}
              \end{cases}
  \]
\end{prop}
\begin{proof}
The first statement follows easily from the equivariant self-intersection formula \cite[p. 621, (4)]{Edidin-Graham-98}, since $[Q] = j_*[Q']$.  Since both $X$ and $Y$ have $T$-actions (not simply $S$-actions, as is the case for $Q$ and $Q'$), there does exist a top $T$-equivariant Chern class of $N_X Y$, and clearly the $S$-equivariant version is simply the restriction of the $T$-equivariant one.

The properties which define $\ga$ follow from analysis of tangent spaces at various fixed points.  Indeed, it is clear that the $T$-fixed points of $Q'$ lying in $Y$ are those which correspond to elements of $W_L \subseteq W$.  Thus for $w \notin W_L$, we have $\ga|_w = 0$.  For $w \in W_L$, the restriction is $c_d^T(N_X Y)|_w = c_d^T(N_X Y|_w) = c_d^T(T_w X/T_w Y)$, where $d=\text{codim}_X(Y)$.  Since both $X$ and $Y$ are flag varieties, it is straightforward to compute these two tangent spaces, and their decompositions as representations of $T$. Indeed, $T_w X$ is simply $\bigoplus_{r \in \Phi^+} \mathfrak{g}_{w r}$, while $T_w Y$ is $\bigoplus_{r \in \Phi_L^+} \mathfrak{g}_{w r}$. The quotient of the two spaces is then $\bigoplus_{r \in \Phi^+ \setminus \Phi_L^+} \mathfrak{g}_{w r}$, which implies our claim on $\ga|_w$.

Finally, that these restrictions determine $\ga$ follows from the localization theorem, Theorem \ref{thm:eqvt-localization}.
\end{proof}

We now determine an explicit formula for the $T$-equivariant Chern class $\ga$ that is defined in Proposition~\ref{prop:eqvt-cohomology-formula-general} when $G$ is of type $A$.
Let $\mu=(\mu_1,\hdots,\mu_s)$ be a composition of $n$, let $\mb{L} = \mb{GL}_{\mu_1} \times \hdots \times \mb{GL}_{\mu_s}$, and let $\mb{T}$ be the full diagonal torus of $\mb{GL}_n$.

For each $1 \leq k \leq n$, let $\epsilon_k \in \mathfrak{t}^*$ be given by $\epsilon_k(\text{diag}(t_1,\dots,t_n)) = t_k$.  For any $1 \leq k < l \leq n$, let $\alpha_{k,l} = \epsilon_k - \epsilon_{\ell} \in \Phi^+$.  Finally, let
$$h_{\mu}(\xx,\yy) := \prod_{\alpha_{k,\ell} \in \Phi^+ \setminus \Phi^+_L} (x_k - y_{\ell}).$$

An equivalent definition of $h_\mu$ showing the explicit dependence on the composition $\mu$ is as follows.  For each pair $i,j$ with $1 \leq i < j \leq s$, we define a polynomial
\[ h_{i,j}(\xx,\yy) := \displaystyle\prod_{k=\nu_j+1}^{\nu_j+\mu_j} \displaystyle\prod_{\ell=\nu_i+1}^{\nu_i+\mu_i} (x_\ell - y_k). \]
Then it follows immediately that
 \[ h_{\mu}(\xx,\yy) = \displaystyle\prod_{1 \leq i < j \leq s} h_{i,j}(\xx,\yy). \]

\begin{prop}\label{prop:formula-for-chern-class}
The $T$-equivariant class $\ga$ is represented by the polynomial $h_{\mu}(\xx,\yy)$.
\end{prop}
\begin{proof}
It is straightforward to verify that the polynomial representative we give satisfies the restriction requirements of Proposition \ref{prop:eqvt-cohomology-formula-general}.  Indeed, the Weyl group $W_{\mb{L}}$ of $\mb{L}$ is a parabolic subgroup of the symmetric group on $n$-letters, embedded as those permutations preserving separately the sets $\{\nu_i+1,\hdots,\nu_i+\mu_i\}$ for $i=0,\hdots,s-1$.
Applying such a permutation to the representative above (with the action by permutation of the indices on the $\xx$-variables, as in Proposition \ref{prop:restriction-maps}) gives the appropriate product of weights. On the other hand, applying any $w \notin W_\mb{L}$ will clearly give $0$, since such a permutation necessarily sends some $l \in \{\nu_i+1,\hdots,\nu_i+\mu_i\}$ (for some $i$) to some $k \in \{\nu_j+1,\hdots, \nu_j+\mu_j\}$ with $i<j$. This permutation forces the factor $x_\ell - y_k$ to vanish, which, in turn, forces $h_{\mu}(\xx,\yy)$ to vanish. By Proposition \ref{prop:eqvt-cohomology-formula-general}, $h_{\mu}$ represents $\ga$.
\end{proof}

\subsection{The orthogonal case}

We now apply these computations specifically to the two type $A$ cases described in Section \ref{sec:bg-examples}, using all of the notational conventions defined there. We start with the case of $(\mb{G},\mb{H})=(\mb{GL}_n,\mb{GO}_n)$.  Let $\mu=(\mu_1,\hdots,\mu_s)$ be a composition of $n$, and let $\mb{B}$, $\mb{T}$, and $\mb{S}_{\mu}$ be as defined in Section \ref{sec:bg-examples}. Let $\mb{P}_{\mu}=\mb{L}_{\mu}\mb{U}_{\mu}$ be the standard parabolic subgroup containing $\mb{B}$ whose Levi factor $\mb{L}_{\mu}$ corresponds to $\mu$.

With these choices made, let the $\xx$, $\yy$, and $\zz$ variables be the generators for $H_{\mb{S}_{\mu}}^*(\mb{G}/\mb{B})$ explicitly described after the statement of Proposition \ref{prop:eqvt-cohom-flag-var}.  We seek a polynomial in the $\xx$, $\yy$, and $\zz$-variables which represents the class of $\mb{H}_{\mu} \cdot 1\mb{B}/\mb{B} \in H_{\mb{S}_{\mu}}^*(\mb{GL}_n/\mb{B})$.

To find such a formula, we use a known formula for the closed $\mb{H}_0=\mb{O}_n$-orbit on $\mb{GL}_n/\mb{B}$ to deduce a formula for the $\mb{S}_{\mu}$-equivariant class of $\mb{H}_{\mu} \cdot 1\mb{B}/\mb{B}$ in $H_{\mb{S}_{\mu}}^*(\mb{L}_{\mu}/\mb{B}_{\mb{L}_{\mu}})$, and then apply Proposition \ref{prop:eqvt-cohomology-formula-general}.  Indeed, we know from \cite{Wyser-13-TG,Wyser-Yong-15+} that when $\mb{H}_0 = \mb{O}_n$, the class of $\mb{H}_0 \cdot 1\mb{B}/\mb{B}$ in $H_{\mb{S}_0}^*(\mb{GL}_n/\mb{B})$ (here, $\mb{S}_0$ is the maximal torus of $\mb{H}_0$ described in Section \ref{sec:bg-examples}) is given by
\begin{equation}\label{eq:orthogonal-s-eqvt-formula}
[\mb{H}_0 \cdot 1\mb{B}/\mb{B}]_{\mb{S}_0} = \displaystyle\prod_{1 \leq i \leq j \leq n - i} (x_i + x_j) = 2^{\lfloor n/2 \rfloor} \displaystyle\prod_{i \leq n/2} x_i \displaystyle\prod_{1 \leq i < j \leq n-i} (x_i + x_j).
\end{equation}

The formula of \eqref{eq:orthogonal-s-eqvt-formula} is given in \cite{Wyser-13-TG} in the case that $n$ is even, while an alternative formula is given in the case that $n$ is odd.  It is observed in \cite{Wyser-Yong-15+} that the above formula applies equally well when $n$ is odd.

Recall that when $\mb{H}=\mb{GO}_n$, a maximal torus $\mb{S}$ of $\mb{H}$ has dimension one greater than the corresponding maximal torus $\mb{S}_0$ of $\mb{H}_0$. Thus the $\mb{S}$-equivariant cohomology of $\mb{GL}_n/\mb{B}$ has one additional ``equivariant variable'', which we call $z$. In this case, it is no harder to show that the class of $\mb{H} \cdot 1\mb{B}/\mb{B}$ in $H_{\mb{S}}^*(\mb{GL}_n/\mb{B})$ is given by
\begin{equation}\label{eq:orthogonal-s-eqvt-formula-new}
  [\mb{H} \cdot 1\mb{B}/\mb{B}]_{\mb{S}} = P_n(\xx,\yy,z) := \displaystyle\prod_{1 \leq i \leq j \leq n-i} (x_i + x_j - 2z).
\end{equation}
Note that by restricting from $H_{\mb{S}}^*(\mb{G}/\mb{B})$ to $H_{\mb{S}_0}^*(\mb{G}/\mb{B})$ (which amounts to setting the additional equivariant variable $z$ to $0$), we recover the original formula \eqref{eq:orthogonal-s-eqvt-formula}.

Combining these formulae with Proposition \ref{prop:eqvt-cohomology-formula-general}, for each composition $\mu$ of $n$ we are now ready to give case-specific formulae for the unique closed $\mb{H}_{\mu}$-orbit $\mb{H}_{\mu} \cdot 1\mb{B}/\mb{B}$ on $\mb{GL}_n/\mb{B}$. (Recall that each of these orbits corresponds to the unique closed $\mb{B}$-orbit on the corresponding $\mb{G}$-orbit on the wonderful compactification of $\mb{G}/\mb{H}$.)

We consider the variables $\xx = (x_1,\dots,x_n)$, $\yy = (y_1,\dots,y_n)$, and $\zz = (z_1,\dots,z_s)$, where $s$ is the number of parts in the composition $\mu$.  We divide the variables into smaller clusters dictated by the composition $\mu$.  Let $\xx^{(i)} = (x_{\nu_i + 1}, \dots, x_{\nu_{i+1}})$ and $\yy^{(i)} = (y_{\nu_i + 1}, \dots, y_{\nu_{i+1}})$.  Then define
\[ P_{\mu}(\xx, \yy, \zz) := \prod_{i=1}^s P_{\mu_i}(\xx^{(i)}, \yy^{(i)}, z_i), \]
where $P_{\mu_i}$ is given by \eqref{eq:orthogonal-s-eqvt-formula-new}.

We now introduce an equivalent, but more explicit, description of the class $[\mb{H}_{\mu} \cdot 1 \mb{B} / \mb{B}]_{S_{\mu}}$ which reflects the block decomposition associated to $\mu$.  To this end, we introduce new notation for a fixed composition $\mu=(\mu_1,\hdots,\mu_s)$ of $n$.  First, for each of the $s$ blocks of $\mb{S}_{\mu}$, define a polynomial $f_i(\xx,\zz)$ as follows:
\[ f_i(\xx,\zz) = \displaystyle\prod_{j=\nu_i+1}^{\nu_i+\lfloor \mu_i/2 \rfloor} (x_j - z_i), \]

In words, the $x_j$ occurring in the terms of this product are those occurring in the first half of their block, and from each, we subtract the $\zz$-variable corresponding to that block.  So for example, if $n=11$, $\mu=(6,5)$, then
\[ f_1(\xx,\zz) = (x_1 - z_1)(x_2 - z_1)(x_3 - z_1), \]
while
\[ f_2(\xx,\zz) = (x_7 - z_2)(x_8 - z_2). \]

Next, for each block, define $g_i(\xx,\zz)$ as follows:
\[ g_i(\xx,\zz) = \displaystyle\prod_{\nu_i+1 \leq j < k \leq 2\nu_i+\mu_i-j} (x_j + x_k - 2z_i). \]

(Note that $g_i(\xx,\zz) = 1$ unless $\mu_i \geq 3$.)  So for $\mu=(6,5)$ as above, we have
\[ g_1(\xx,\zz) = (x_1 + x_2 - 2z_1)(x_1 + x_3 - 2z_1)(x_1 + x_4 - 2z_1)(x_1 + x_5 - 2z_1)(x_2 + x_3 - 2z_1)(x_2 + x_4 - 2z_1), \]
and
\[ g_2(\xx,\zz) = (x_7 + x_8 - 2z_2)(x_7 + x_9 - 2z_2)(x_7 + x_{10} - 2z_2)(x_8 + x_9 - 2z_2). \]

Finally, we define a third polynomial $h_{\mu}(\xx,\yy,\zz)$ in the $\xx$, $\yy$, and $\zz$-variables to simply be $h_{\mu}(\xx,\rho(\yy))$, where $\rho$ denotes restriction from the variables $y_1,\hdots,y_n$ corresponding to coordinates on the full torus $\mb{T}$ to the variables $y_{i,j},z_i$ on the smaller torus $\mb{S}_{\mu}$.  To be more explicit, for each $i,j$
with $1 \leq i < j \leq s$ define
\[
 h_{i,j}(\xx,\yy,\zz) :=
\begin{cases}
 \displaystyle\prod_{k=1}^{\mu_i} \displaystyle\prod_{l=1}^{\mu_j/2} (x_{\nu_i+k} - y_{j,l} - z_j)(x_{\nu_i+k} + y_{j,l} - z_j)
& \text{ if $\mu_j$ is even} \\
  \displaystyle\prod_{k=1}^{\mu_i} (x_{\nu_i+k} - z_j) \displaystyle\prod_{l=1}^{\lfloor \mu_j/2 \rfloor} (x_{\nu_i+k} - y_{j,l} - z_j)(x_{\nu_i+k} + y_{j,l} - z_j)
& \text{ if $\mu_j$ is odd}
\end{cases}
\]

So for the case $n=4$, $\mu=(2,2)$, we have
\[ h_{1,2}(\xx,\yy,\zz) = (x_1 - y_{2,1} - z_2)(x_1 + y_{2,1} - z_2)(x_2 - y_{2,1} - z_2)(x_2 + y_{2,1} - z_2), \]
while for the case $n=5$, $\mu=(2,3)$,
\[ h_{1,2}(\xx,\yy,\zz) = (x_1 - z_2)(x_2 - z_2)(x_1 - y_{2,1} - z_2)(x_1 + y_{2,1} - z_2)(x_2 - y_{2,1} - z_2)(x_2 + y_{2,1} - z_2).  \]

Then we define
\[ h_{\mu}(\xx,\yy,\zz) = \displaystyle\prod_{1 \leq i < j \leq s} h_{i,j}(\xx,\yy,\zz). \]

Propositions \ref{prop:eqvt-cohomology-formula-general} and \ref{prop:formula-for-chern-class} then imply the following formula for the $\mb{S}_{\mu}$-equivariant class of $\mb{H}_{\mu} \cdot 1\mb{B}/\mb{B}$ in this case.
\begin{corollary}\label{cor:orthogonal-formula-eqvt}
The $\mb{S}_{\mu}$-equivariant class of the unique closed $\mb{H}_{\mu}$-orbit $\mb{H}_{\mu} \cdot 1\mb{B}/\mb{B}$ on $\mb{G}/\mb{B}$ is represented by
the polynomial
\[ 2^{d(\mu)} h_{\mu}(\xx,\yy,\zz) \displaystyle\prod_{i=1}^s f_i(\xx,\zz)g_i(\xx,\zz), \]
where $d(\mu) = \displaystyle\sum_{i=1}^s \lfloor \mu_i/2 \rfloor$.
\end{corollary}
\begin{proof}
The fact that the product $2^{d(\mu)} \displaystyle\prod_{i=1}^s f_i(\xx,\zz)g_i(\xx,\zz)$ is equal to the formula for $j^*[Q]$ follows from the formula of \eqref{eq:orthogonal-s-eqvt-formula-new}.  It follows from
Proposition \ref{prop:eqvt-cohomology-formula-general} and Proposition \ref{prop:formula-for-chern-class} that the representative of $[Q]$ is obtained from $j^*[Q]$ by multiplying with the top $S_{\mu}$-equivariant Chern class of the normal bundle, which is represented by the polynomial $h_{\mu}(\xx,\yy,\zz)$.
\end{proof}

\begin{corollary}
The $\mb{S}_{\mu}$-equivariant class $[\mb{H}_{\mu} \cdot 1 \mb{B} / \mb{B}]$ is represented by the polynomial $P_{\mu}(\xx,\yy,\zz) h_{\mu}(\xx,\yy,\zz)$.
\end{corollary}
\begin{proof}
This follows immediately from the observation that $$P_{\mu_i}(\xx^{(i)}, \yy^{(i)}, z_i) = 2^{\lfloor \mu_i /2 \rfloor} f_i(\xx,\zz) g_i(\xx,\zz).$$
\end{proof}

\subsection{The symplectic case}

We give similar (but simpler) formulas for the case when $(\mb{G},\mb{H'}) = (\mb{GL}_{2n},\mb{GSp}_{2n})$.  Recall from \cite{Wyser-13-TG} that for the case $(\mb{G},\mb{H}_0') = (\mb{GL}_{2n},\mb{Sp}_{2n})$, the $\mb{S}_0'$-equivariant class ($\mb{S}_0'$ the maximal torus of $\mb{H}_0'$ described in Section \ref{sec:bg-examples}) of the unique
closed orbit $\mb{H}_0' \cdot 1 \mb{B}/\mb{B}$ is given by
\begin{equation}\label{eq:symplectic-base-formula}
[\mb{H}_0' \cdot 1\mb{B}/\mb{B}]_{\mb{S}_0'} = \displaystyle\prod_{1 \leq i < j \leq 2n-i} (x_i + x_j).
\end{equation}

As before, it is no harder to see that if $\mb{S}^\prime$ is the maximal torus of diagonal elements of $\mb{H}^\prime$, then the $\mb{S}^\prime$-equivariant class of the closed orbit $\mb{H}^\prime \cdot 1\mb{B}/\mb{B}$ is represented by
\begin{equation}\label{eq:symplectic-base-formula-with-z}
[\mb{H}^\prime \cdot 1\mb{B}/\mb{B}]_{\mb{S}^\prime} = P'_n(\xx,\yy,z) := \displaystyle\prod_{1 \leq i < j \leq 2n-i} (x_i + x_j - 2z).
\end{equation}

Now let $\mu=(\mu_1,\hdots,\mu_s)$ be a composition of $2n$ with all even parts. Let $\mb{H}_{\mu}'$ be the spherical group
\[ (\mb{GSp}_{\mu_1} \times \hdots \times \mb{GSp}_{\mu_s}) \ltimes \mb{U}_{\mu}, \]
as defined in Section \ref{sec:bg-examples}.  Let $\mb{S'}_{\mu}$ be the maximal torus of $\mb{H}_{\mu}'$, with the $\yy$ and $\zz$-variables as defined above.

\begin{corollary}\label{cor:symplectic-formula-eqvt}
In $H_{\mb{S'}_{\mu}}^*(\mb{G}/\mb{B})$, the class of the closed $\mb{H}_{\mu}'$-orbit $\mb{H}_{\mu}' \cdot 1\mb{B}/\mb{B}$ is represented by
\[ h_{\mu}(\xx,\yy,\zz) \displaystyle\prod_{i=1}^s g_i(\xx,\zz). \]
\end{corollary}
\begin{proof}
The proof is identical to that of Corollary \ref{cor:orthogonal-formula-eqvt}, using Proposition \ref{prop:eqvt-cohomology-formula-general} combined with \eqref{eq:symplectic-base-formula-with-z} in the same way.
\end{proof}

\begin{remark}
Note that $g_i(\xx,\zz) = P'_{\mu_i}(\xx^{(i)}, \yy^{(i)}, z_i)$, where $P'_{\mu_i}$ is given by \eqref{eq:symplectic-base-formula-with-z}, so that again $[\mb{H}_{\mu}' \cdot 1 \mb{B} / \mb{B}]$ is equal to $P'_{\mu}(\xx,\yy, \zz) h_{\mu}(\xx,\yy,\zz)$, where $$P'_{\mu}(\xx,\yy,\zz) := \prod_{i=1}^s P'_{\mu_i}(\xx^{(i)}, \yy^{(i)}, z_i).$$ \qed
\end{remark}

We now specialize these formulas to ordinary cohomology (by setting all $\yy$ and $\zz$-variables to $0$), in order to prove Corollaries \ref{cor:orthogonal-formula-ordinary} and \ref{cor:symplectic-formula-ordinary}.
First, we define the notations used in those formulas which have not yet been defined. For each $i=1,\hdots,n$, let $B(\mu,i)$ denote the block that the variable $x_i$ occurs in, i.e. $B(\mu,i)$ is the smallest integer $j$ such that
\[ \displaystyle\sum_{\ell=1}^j \mu_\ell \geq i. \]

Then for each $i=1,\hdots,n$, define $R(\mu,i)$ to be
\[ R(\mu,i) := \displaystyle\sum_{B(\mu,i) < j \leq s} \mu_j. \]
This is the combined size of all blocks occurring strictly to the right of the block in which $x_i$ occurs.

Finally, again for each such $i$, define $\delta(\mu,i)$ to be $1$ if and only if $x_i$ occurs in the first half of its block, and $0$ otherwise.  Note that by the ``first half'' we mean those positions less than or equal to $\ell/2$ where $\ell$ is the size of the block; in particular, for a block of odd size, the middle position is not considered to be in the first half of the block.

\begin{proof}[Proof of Corollaries \ref{cor:orthogonal-formula-ordinary} and \ref{cor:symplectic-formula-ordinary}]
The formula of Corollary \ref{cor:orthogonal-formula-ordinary} comes from that of Corollary \ref{cor:orthogonal-formula-eqvt}; we simply set $\yy=\zz=0$.  The binomial terms $x_j + x_k$ come from the polynomials $g_i(\xx,0)$.  The monomial terms come from the polynomials $f_i(\xx,0)$ and $h_{i,j}(\xx,0,0)$.  The $x_i^{\delta(\mu,i)}$ term comes from $f_i(\xx,0)$, the latter being $x_i$ if this variable occurs in the first half of its block, and $1$ otherwise.  The remaining $x_i^{R(\mu,i)}$ comes from the $h_{i,j}(\xx,0,0)$.  Indeed, it is evident that for an $\xx$-variable in block $i$, for each $j > i$ the given $\xx$-variable appears in precisely $\mu_j$ linear forms involving $\yy,\zz$ terms associated with block $j$.

The proof of Corollary \ref{cor:symplectic-formula-ordinary} is almost identical, except simpler.
\end{proof}

\section{Factoring sums of Schubert polynomials}\label{sec:schubert-polys}
We end by establishing explicit polynomial identities involving sums of Schubert polynomials, using the cohomological formulae of the preceding section together with the results of \cite{Can-Joyce,Can-Joyce-Wyser-15} which were recalled in Section \ref{sec:bg-w-sets}.

Note that by Brion's formula \eqref{eq:brion-formula} combined with the fact that the Schubert polynomial $\frS_w$ is a representative of the class of the Schubert variety $\mb{X}_w$ in $H^*(\mb{G}/\mb{B})$, we have the following two families of identities in $H^*(\mb{G}/\mb{B})$:

\begin{equation}\label{eq:schubert-identity-orthogonal-version-1}
\displaystyle\sum_{w \in W(\mb{Y}_{\mu})} 2^{d(\mb{Y}_{\mu},w)} \frS_w = 2^{d(\mu)} \displaystyle\prod_{i=1}^n x_i^{R(\mu,i) + \delta(\mu,i)} \displaystyle\prod_{i=1}^s \left( \displaystyle\prod_{\nu_i+1 \leq j \leq k \leq \nu_{i+1} - j} (x_j + x_k) \right);
\end{equation}

\begin{equation}\label{eq:schubert-identity-symplectic}
\displaystyle\sum_{w \in W(\mb{Y}_{\mu}')} \frS_w = \displaystyle\prod_{i=1}^{2n} x_i^{R(\mu,i)} \displaystyle\prod_{i=1}^s \left( \displaystyle\prod_{\nu_i+1 \leq j < k \leq \nu_{i+1} - j} (x_j + x_k) \right).
\end{equation}

Equation \eqref{eq:schubert-identity-orthogonal-version-1} above simply combines \eqref{eq:brion-formula} with Corollary \ref{cor:orthogonal-formula-ordinary}.  Likewise, \eqref{eq:schubert-identity-symplectic} combines \eqref{eq:brion-formula} with Corollary \ref{cor:symplectic-formula-ordinary}, together with the fact that all $\mb{B}$-orbit closures in the symplectic case are known to be multiplicity-free, meaning $d(\mb{Y}_{\mu}',w) = 0$ for all $w \in W(\mb{Y}_{\mu}')$.

In fact, also in \eqref{eq:schubert-identity-orthogonal-version-1} above, the powers of $2$ can be completely eliminated from both sides of the equation.  This follows from a result of Brion \cite[Proposition 5]{Brion-01}, which states that whenever $G$ is a simply laced group (recall that for us, $\mb{G} = \mb{GL}_n$ is simply laced), all of the coefficients appearing in \eqref{eq:brion-formula} are the \textit{same} power of $2$.  It is explained in \cite[Section 5]{Can-Joyce} that the coefficients appearing on the left-hand side of \eqref{eq:schubert-identity-orthogonal-version-1} are in fact all equal to $2^{d(\mu)}$.  Since $H^*(\mb{G}/\mb{B})$ has no torsion, we have the simplified equality
\begin{equation}\label{eq:schubert-identity-orthogonal-version-2}
\displaystyle\sum_{w \in W(\mb{Y}_{\mu})} \frS_w = \displaystyle\prod_{i=1}^n x_i^{R(\mu,i) + \delta(\mu,i)} \displaystyle\prod_{i=1}^s \left( \displaystyle\prod_{\nu_i+1 \leq j \leq k \leq \nu_{i+1} - j} (x_j + x_k) \right).
\end{equation}

Now, note that \textit{a priori}, the identities \eqref{eq:schubert-identity-symplectic} and \eqref{eq:schubert-identity-orthogonal-version-2} hold only in $H^*(\mb{G}/\mb{B})$.  That is, we know only that the left and right-hand sides of the identities are congruent modulo the ideal $I^{\mb W}$.  We end with a stronger result.

\begin{theorem}\label{thm:sum-equals-product}
Both \eqref{eq:schubert-identity-symplectic} and \eqref{eq:schubert-identity-orthogonal-version-2} are valid as polynomial identities.
\end{theorem}
\begin{proof}
We use the fact that the Schubert polynomials $\{\frS_w \mid w \in S_n\}$ are a $\Z$-basis for the $\Z$-submodule $\Gamma$ of $\Z[\xx]$ spanned by monomials $\prod x_i^{c_i}$ with $c_i \leq n-i$ for each $i$ \cite[Proposition 2.5.4]{Manivel-Book-01}.  We claim that on the right-hand side of \eqref{eq:schubert-identity-orthogonal-version-2} (resp. \eqref{eq:schubert-identity-symplectic}), each $x_i$ does in fact occur with exponent at most $n-i$ (resp. $2n-i$).  To see this, note that since the right-hand side of each identity is a product of linear forms, it suffices to count, for each $i$, the number of these linear factors in which $x_i$ appears.

In \eqref{eq:schubert-identity-orthogonal-version-2}, we claim that $x_i$ appears in precisely $\left( \displaystyle\sum_{j=B(\mu,i)}^s \mu_j \right)-i$ of the linear factors.  Since $\displaystyle\sum_{j=B(\mu,i)}^s \mu_j \leq \displaystyle\sum_{j=1}^s \mu_j = n$, this establishes our claim that the right-hand side lies in $\Gamma$.  Indeed, clearly $x_i$ appears $R(\mu,i)+\delta(\mu,i)$ times as a monomial factor, so we need only count the number of binomial factors of the form $x_j + x_k$ that it appears in.  One checks easily that it appears in $\mu_i - i - 1$ such factors if $x_i$ occurs in the first half of its block, and in $\mu_i - i$ such factors otherwise.  In other words, if $N_i$ is the number of binomial factors involving $x_i$, then we have $\delta(\mu,i) + N_i = \mu_i-i$.  Thus
\[
\begin{array}{lll}
R(\mu,i) + \delta(\mu,i) + N_i & = & R(\mu,i) + \mu_i-i \\
& = & \left( \displaystyle\sum_{j=B(\mu,i)+1}^s \mu_j \right) + \mu_i - i \\
& = & \left( \displaystyle\sum_{j=B(\mu,i)}^s \mu_j \right) - i,
\end{array}
\]
as claimed.

Clearly, the right-hand side of \eqref{eq:schubert-identity-symplectic} also lies in $\Gamma$, applying the same argument with $n$ replaced by $2n$.  Indeed, the only difference is in the lack of the additional
monomial factor $x_i^{\delta(\mu,i)}$; thus $x_i$ occurs in either $\left( \displaystyle\sum_{j=B(\mu,i)}^s \mu_j \right) - i$ or $\left( \displaystyle\sum_{j=B(\mu,i)}^s \mu_j \right) - i - 1$ of the linear factors on the right-hand side of \eqref{eq:schubert-identity-symplectic}.  In either event, this is at most $2n-i$, as required.

Now, since the right-hand side of each of \eqref{eq:schubert-identity-symplectic} and \eqref{eq:schubert-identity-orthogonal-version-2} are in $\Gamma$, they are expressible as a sum of Schubert polynomials whose indexing permutations lie in $S_{2n}$ (for \eqref{eq:schubert-identity-symplectic}) or $S_n$ (for \eqref{eq:schubert-identity-orthogonal-version-2}) in exactly one way.  Furthermore, since the Schubert classes $\{[\mb{X}_w]\}$ are a $\Z$-basis for $H^*(\mb{G}/\mb{B})$, the cohomology class represented by the right-hand side of \eqref{eq:schubert-identity-symplectic} and \eqref{eq:schubert-identity-orthogonal-version-2} is a $\Z$-linear combination of Schubert classes in precisely one way.  Clearly, the same indexing permutations must arise with the same multiplicities in both the polynomial expansion and the cohomology expansion.  Then since \eqref{eq:schubert-identity-symplectic} and \eqref{eq:schubert-identity-orthogonal-version-2} are correct cohomologically, they must also be polynomial identities.
\end{proof}

\begin{example}
In the orthogonal case when $\mu = (3,4)$, the identity \eqref{eq:schubert-identity-orthogonal-version-2} becomes
\begin{align*}
\mathfrak{S}_{6752431} + \mathfrak{S}_{6753412} + \mathfrak{S}_{6754213} + & \mathfrak{S}_{7562431} + \mathfrak{S}_{7563412} + \mathfrak{S}_{7564213} = \\ & x_1^5 x_2^4 x_3^4 x_4 x_5 (x_1 + x_2)(x_4 + x_5)(x_4 + x_6).
\end{align*}
\end{example}

\section*{Acknowledgments}
We thank Michel Brion for many helpful conversations and suggestions. We thank the referee for their careful reading and many helpful suggestions.  The first and second author were supported by NSA-AMS Mathematical Sciences Program grant H98230-14-1-142.  The second is partially supported by the Louisiana Board of Regents Research and Development Grant 549941C1.  The third author was supported by NSF International Research Fellowship 1159045 and hosted by Institut Fourier in Grenoble, France.

\bibliographystyle{plain}
\bibliography{sourceDatabase}
\end{document}